\documentclass[11pt]{article}
\usepackage{amssymb}
\usepackage[latin1]{inputenc}
\usepackage{graphicx,color}
\def\nd{\noindent}

\usepackage{color}
\newtheorem{theorem}{Theorem}[section]

\newtheorem{lemma}{Lemma}[section]

\newcommand{\fim}{\hfill\rule{2mm}{2mm}}

\date{}
\begin{document}
\title{
\vspace{0.5in} {\bf\Large  Infinite many Blow-up solutions for  a  Schr\"odinger quasilinear elliptic problem with a non-square diffusion term}}

\author{
{\bf\large Carlos Alberto Santos}\footnote{Carlos Alberto Santos acknowledges
the support of CAPES/Brazil Proc.  $N^o$ $2788/2015-02$,}\,\, ~~~~~~ {\bf\large Jiazheng Zhou}\footnote{Jiazheng Zhou was supported by CNPq/Brazil Proc. $N^o$ $232373/2014-0$, }\ \ \ 
\hspace{2mm}\\
{\it\small Universidade de Bras\'ilia, Departamento de Matem\'atica}\\
{\it\small   70910-900, Bras\'ilia - DF - Brazil}\\
{\it\small e-mails: csantos@unb.br,
jiazzheng@gmail.com }\vspace{1mm}\\
}

\date{}
\maketitle \vspace{-0.2cm}

\begin{abstract}
In this paper, we consider existence of positive solutions for the Schr\"odinger quasilinear elliptic problem
$$
  \left\{
\begin{array}{l}
\Delta_pu+\Delta_p(|u|^{2\gamma})|u|^{2\gamma-2}u  = a(x)g(u)~ \mbox{on}~ \mathbb{R}^N,\\
u>0\  \mbox{in}~\mathbb{R}^N,\  u(x)\stackrel{\left|x\right|\rightarrow \infty}{\longrightarrow} \infty,
\end{array}
\right.
$$
where $a(x), ~x\in \mathbb{R}^N$ and $g(s)~s>0$ are a nonnegative and continuous functions with $g$ being nonincreasing as well,  $\gamma>{1}/{2}$, and $N \geq 1$. By a dual approach we establish sufficient conditions for existence and multiplicity of solutions for this problem.
\end{abstract}

\nd {\it \footnotesize 2012 Mathematics Subject Classifications:} {\scriptsize  35J10,  35J92,  35B07, 35B08, 35B44.}\\
\nd {\it \footnotesize Key words}: {\scriptsize Schr\"odinger equations, Blow up solutions, Quasilinear problems, Non-square diffusion, Multiplicity of solutions.}

\section{Introduction}
\def\theequation{1.\arabic{equation}}\makeatother
\setcounter{equation}{0}

In this paper, let us consider the problem
\begin{equation}\label{P}
 \left\{
\begin{array}{l}
\Delta_pu+\Delta_p(|u|^{2\gamma})|u|^{2\gamma-2}u  = a(x)g(u)\quad \mbox{on } \mathbb{R}^N,\\
u>0\  \mbox{in}~\mathbb{R}^N,\  u(x)\stackrel{\left|x\right|\rightarrow \infty}{\longrightarrow} \infty,
\end{array}
\right.
\end{equation}
where $\Delta_pu=div(|\nabla u|^{p-2}\nabla u)$ with $1<p<\infty$ is called the $p-Laplacian$ operator, $a(x),~x \in \mathbb{R}^N$ is a nonnegative continuous function, $g(s),~s\geq0$ is a nondecreasing continuous function that satisfies $g(0)=0$,   $\gamma>1/2$, and  $N\geq 1$.

In the case $p=2$, the equation (\ref{P}) is referred in the literature as a modified nonlinear Scho\"odinger equation because it contains a quasilinear and nonconvex term $\Delta(|u|^{2\gamma})|u|^{2\gamma-2}u$. This quasilinear term  is called of non-square diffusion for $\gamma\neq 1$ and square diffusion for $\gamma=1$. In particular when $\gamma=1$, the solution of (\ref{P}) is related to standing wave solutions for the quasilinear Schr$\ddot{o}$dinger equation
\begin{equation}\label{i}
 iz_t+\Delta z-\omega(x)z+\kappa\Delta(h(|z|^2))h'(|z|^2)z+\eta(x,z)=0,\ \ x\in\mathbb{R}^N,
\end{equation}
where $\omega$ is a potential given , $h$ and $\eta$ are real functions and $\kappa$ is a real constant. This connecting is established  by the fact that $z(t,x)=e^{-i\beta t}u(x)$ is a solution to the 
 equation (\ref{i}),  if $u$ satisfies the equation in  (\ref{P}), for suitable $\omega$, $h$ and $\eta$.

The quasilinear Schr$\ddot{o}$dinger equation (\ref{i}) is an important model that comes from of several mathematical and physical phenomena,
for example, if $h(s)=s$ it models a superfluid film in plasma physics \cite{kurihara}, while for $h(s)=(1+s)^{1/2}$, the equation (\ref{i}) models the self-channeling of a high-power ultrashort laser in matter \cite{kosevich} and \cite{quispel}. In addition, it also appears in the theory of Heidelberg ferromagnetism and magnus \cite{hasse},\cite{makhankov}; in dissipative quantum mechanics 
\cite{adachi}; and in condensed matter theory \cite{byeon}. 

On the other side, it is well-known that the blow-up condition appears  in the study of population dynamics, subsonic motion of a gas, non-Newtonian fluids, non-Newtonian filtration as well as in the theory of the electric potential in a glowing hollow metal body. The research of this subject  passed by a great development with the works of   Keller \cite{K} and Osserman
 \cite{O} in 1957 that established necessary and sufficient conditions for existence of solutions for the semilinear and autonomous problem (that is, $p=2$ and $a \equiv 1$)
 \begin{equation}
 \label{P1}
 \left\{
\begin{array}{l}
 \Delta_p u=a(x)g(u)\ \ \mbox{in}\ \mathbb{R}^N,\\
u\geq 0\  \mbox{in}\ \mathbb{R}^N,\  u(x)\stackrel{\left|x\right|\rightarrow \infty}{\longrightarrow} \infty,
\end{array}
\right.
 \end{equation}
where $g$ is a non-decreasing  continuous function. Keller established that

\noindent$
{\bf(G):}~~~~~~\displaystyle\int_1^{+\infty} \frac{dt}{\sqrt[p]{G(t)}}=\infty,~\mbox{where}~G(t)=\int_0^tg(s)ds,~t>0
$

\noindent is a sufficient and necessary condition for the problem $(\ref{P1})$ to have solution. In this same year, Osserman proved the same result for sub solutions 
of $(\ref{P1})$. 
After these works, when {\bf(G)} is not satisfied, the function $g$ has become well-known as a $Keller-Osserman$ function.

After this, a number of researchers have worked in related problems. These researches have showed that the existence of solutions for (\ref{P1}) is very sensible to the ``how radial" is $a(x)$ at infinity, that is, how big is the number 
$$a_{osc}(r):= \overline{a}(r) - \underline{a}(r),~r\geq 0,$$
where
\begin{equation}
\label{defa}
\underline{a}(r)=\min\{a(x)~/~ |x|= r\}~~\mbox{and}~~ \overline{a}(r) =\max\{a(x)~/~ |x| = r\},~r\geq 0.
\end{equation}

Note that $a_{osc}(r)= 0$, $r\geq r_0$ if, and only if, $a$ is symmetric radially in $\vert x \vert \geq r_0$, for some $r_0\geq 0$. That is, if $r_0=0$ we say that $a$ is radially symmetric.

In this direction, for $p=2$ and $a$ radially symmetric, Lair and Wood in \cite{lairwood} considered  $g(u)=u^{\gamma}$, $u\geq 0$ with $0 < \gamma \leq 1$ (that is, $g$ does not satisfies $\bf{(G)})$ and showed that 
\begin{equation}
\label{conda}
\int_1^{\infty} r a(r) dr = \infty.
\end{equation}
is a sufficient condition for (\ref{P1}) has a radial solution.

Still with $p=2$, in 2003, Lair \cite{lair3} enlarged the class of potentials $a(x)$ by permitting $a_{osc} $ to assume values not identically null, but not too big ones. More exactly, he assumed
$$\int_0^{\infty}r a_{osc}(r) exp(\underline{A}(r)) dr< \infty,~~\mbox{where}~~ \underline{A}(r)= \int_0^r s\underline{a}(s) ds,~r\geq 0 
$$
and proved that (\ref{P1}), with suitable $f$ that includes $u^{\gamma}$ for $0< \gamma \leq 1$, admits a solution if, and only if, (\ref{conda}) holds with $\underline{a}$ in the place of $a$. Keeping us in this context, we also quote Rhouma and Drissi \cite{drissi} for $2 \leq p \leq N$, and references therein.

Coming back to problem (\ref{P}), we note that issues about existence and multiplicity of solutions for equations related to the equation in (\ref{P}) (since positive, negative to nodal solutions)  have been treated by a number of researchers recently, but there is no accurate results for existence of solutions to  (\ref{P}), that is, with the blow up behavior for the solutions. See for instance \cite{lww, wu, dly, eg100, eg, kw, jxw} and references therein.

Before stating ours principal results, we set that a solution of (\ref{P}) is a positive function 
 $u\in C^1(\mathbb{R}^N)$ that satisfies  $u \to \infty$ as $\vert x \vert \to \infty$ and
 $$
  \left.
\begin{array}{l}
\displaystyle\int_{\mathbb{R}^N}(1+(2\gamma)^{p-1}u^{p(2\gamma-1)})|\nabla u|^{p-2}\nabla u\nabla\varphi dx~+~~~~~~~~~~\\
\\
~~~~~~~~~~~~~~~~\displaystyle(2\gamma)^{p-1}(2\gamma-1)\int_{\mathbb{R}^N}|\nabla u|^pu^{p(2\gamma-1)-1}\varphi 
dx+\int_{\mathbb{R}^N}a(x)g(u)\varphi dx=0,~\mbox{for all}~ \varphi\in C_0^{\infty}(\mathbb{R}^N),
\end{array}
\right.
 $$
 and consider the assumption on $g$
\smallskip

\noindent$
{\bf(g):}~~~~~~\displaystyle\liminf_{t\to\infty}\frac{g(t)}{t^{2\gamma(2\gamma-1)}}>0.
$
\smallskip

\noindent holds.

We point out that this hypothesis  is so natural, because when $\gamma=1/2$ (that is, the problem (\ref{P}) reduces to (\ref{P1})), it reduces to a standard condition for (\ref{P1})). Our first result is.

\begin{theorem}\label{THRN}
Assume that $g$ satisfies {\bf(g)} and {\bf(G)}. If $a(x)$ is such that $a_{osc}\equiv 0$ and 
\begin{equation}\label{a}
 \int_0^\infty\Big(s^{1-N}\int_0^st^{N-1}a(t)dt\Big)^{\frac{1}{p-1}}ds=\infty
\end{equation}
holds, then there exists a positive constant $\mathcal{A}$ such that $\mathcal{A}_a=(\mathcal{A},\infty)$, 
where
$$\mathcal{A}_a=\{\alpha>\mathcal{A}~/~ (\ref{P})~ \mbox{admits a radial solution with}\ u(0)=\alpha\}.$$   
\end{theorem} 

For non-radial potentials $a(x)$, motivated by recent works, we assume that $\mathcal{G}:(0,\infty) \to (0,\infty)$ defined by
$$\mathcal{G}(t)=\frac{t}{2}g(t)^{{-1}/({p-1})},~t>0,$$ is a non-decreasing and invertible function such that
\begin{enumerate}
\item [\bf{($\mathcal{G}$):}]  $0\leq\overline{H}:=\displaystyle\int_0^\infty\Big(s^{1-N}\int_0^s t^{N-1}a_{osc}(t)dt \Big)^{1/p-1}\Big [g\Big(\mathcal{G}^{-1}
\Big(s\Big(\int_0^s\overline{a}(t)dt\Big)^{1/p-1}\Big)\Big)\Big]^{1/p-1}ds<\infty$
\end{enumerate}
holds.

After this, we state our second result.

\begin{theorem}\label{T2}
Assume $p\geq 2$, $g$ satisfies  {\bf(g)}, {\bf(G)} and $g(t)/t^\delta, t>0$ is non-decreasing for some $\delta\geq2\gamma-1$.  Suppose also that  $a(x)$  is such that $\underline{a}$ satisfies (\ref{a}) and $\overline{a}$ satisfies $\bf{(\mathcal{G})}$. Then  there exists a solution $u\in C^1(\mathbb{R}^N)$ of the problem $(\ref{P})$ satisfying 
$\alpha\leq u(0) \leq (\alpha + \varepsilon)+ \overline{H}$, for each $\alpha>\mathcal{A},\varepsilon>0$ given.
\end{theorem}

We organized this paper in the following way. In the  section 2, we establish an equivalent problem to the (\ref{P}), via an very specific changing variable, and in the last section we completed the proof of theorems 1.1 and 1.2.

\section{Auxiliar results}

We begin this section proving a result that permits us to transform  (\ref{P}) into a new problem with an structure in what is clearest to see how the $Keller-Osserman$ condition works. As noted before, this condition is fundamental  to show existence of solutions that blow-up at infinity. This approach of changing the (\ref{P}) for another one was introduced by \cite{lww} (for $\gamma=1$ and $p=2$) and followed by a number of authors to study related equations to (\ref{i}).

To do this, motivated by \cite{dly} with $p=2$ and \cite{uber} with $\gamma=1$, we are going to consider $f$ given by the solution of the equation
\begin{equation}
\label{mv}
f'(t)=\frac{1}{[1+(2\gamma)^{p-1}|f(t)|^{p(2\gamma-1)}]^{{1}/{p}}},\ t\in(0,\infty);\ f(t)=-f(-t),\ t\in(-\infty,0],
\end{equation}
and we are able to prove the next Lemma.
\begin{lemma}\label{lemma1} Assume $p>1$ and $\gamma>1/2$ hold. Then $f$ satisfies:
\begin{enumerate}
\item [$(f)_1$] $f\in C^1$ is uniquely defined and invertible, 
\item [$(f)_2$] $0<f'(t)\leq1$ for all $t\in\mathbb{R}$,
\item [$(f)_3$] $|f(t)|\leq |t|$ for all $t\in\mathbb{R}$,
\item [$(f)_4$]  ${f(t)}/{t}\to1$ when $t\to 0$,
\item [$(f)_5$] $|f(t)|^{2\gamma}\leq (2\gamma)^{1/p}|t|$ for all $t\in\mathbb{R}$,
\item [$(f)_6$]  ${f(t)}/{2}\leq\gamma tf'(t)\leq\gamma f(t)$ for all $t\geq0$,
\item [$(f)_7$] ${|f(t)|}/{|t|^{1/2\gamma}}\to A>0$ when $|t|\to\infty$, where $A$ is a constant,
\item [$(f)_8$] there exists a positive constant $C$ so that
$|t|\leq C[|f(t)|+|f(t)|^{2\gamma}]$ for all $t\in\mathbb{R}$,
\item [$(f)_9$] $|f(t)|^{(2\gamma-1)}f'(t)\leq1/2^{{(p-1)}/{p}}$ for all $t\in\mathbb{R}$,
\item [$(f)_{10}$] $f'(t)f(t)^\delta$ is non-decreasing for all $t\geq0$ and $\delta\geq2\gamma-1$.
\end{enumerate}
\end{lemma}
\begin{proof}$of~(f)_1:$  Considering  the  problem 
\begin{equation}\label{ode}
   \left\{
\begin{array}{lc}
y'={[1+(2\gamma)^{p-1}|y|^{p(2\gamma-1)}]^{-1/p}},& t>0,\\
y(0)=0,&
 \end{array}
\right.
\end{equation}
it follows from Theorem of existence and uniqueness for initial value problem in ordinary differential equations that the problem (\ref{ode}) has an unique solution, namely, $y=f(t)$. Besides this, $f'(t)>0$ for all $t\in\mathbb{R}$ implies that $f$ is invertible.

\noindent $Proof~of~(f)_2:$ It follows from above that $f'(t)={[1+(2\gamma)^{p-1}|f(t)|^{p(2\gamma-1)}]^{{-1}/{p}}}\in(0,1],\ t\in[0,\infty)$.\\
Now, for $t\leq 0$, it follows from the definition of $f$ that
$f'(t)=f'(-t)={[1+(2\gamma)^{p-1}|f(-t)|^{p(2\gamma-1)}]^{{-1}/{p}}}\in(0,1],$ as well.

\noindent $Proof~of~(f)_3:$ This  follows from $(f)_2$, $f(0)=0$, and  $f'(t)\leq 1$ for all $t \geq 0$, together with the fact of $f$ being a odd function.

\noindent $Proof~of~(f)_4:$ Since, $$f(t)=\int_0^t\frac{ds}{[1+(2\gamma)^{p-1}|f(s)|^{p(2\gamma-1)}]^{{1}/{p}}}=
\frac{t}{[1+(2\gamma)^{p-1}|f(\eta)|^{p(2\gamma-1)}]^{{1}/{p}}},$$
for some $\eta\in(0,t)$, we obtain from this information, that 
$$\lim_{t\to0}\frac{f(t)}{t}=\lim_{\eta\to0}\frac{1}{[1+(2\gamma)^{p-1}|f(\eta)|^{p(2\gamma-1)}]^{{1}/{p}}}=1.$$

\noindent $Proof~of~(f)_5:$ Since, $f'(t)[1+(2\gamma)^{p-1}|f(t)|^{p(2\gamma-1)}]^{{1}/{p}}=1,$ for all $t \in \mathbb{R}$, it follows by integration that
$$\int_0^tf'(s)[1+(2\gamma)^{p-1}|f(s)|^{p(2\gamma-1)}]^{\frac{1}{p}}ds=t,\ \mbox{for}\ t>0.$$

So, doing the change variable $z=f(s)$, we obtain
$$
 t=\int_0^{f(t)}[1+(2\gamma)^{p-1}z^{p(2\gamma-1)}]^{1/p}dz
\geq(2\gamma)^{{(p-1)}/{p}}\int_0^{f(t)}z^{(2\gamma-1)}dz
=(2\gamma)^{{(p-1)}/{p}}\frac{f(t)^{2\gamma}}{2\gamma},
$$
that is, $f(t)^{2\gamma}\leq (2\gamma)^{1/p}t$, for  all $t\geq0$. Since, $f$ is an odd function, it follows the claim $(f)_5$. 

\noindent $Proof~of~(f)_6:$ Let us define  $F_1(t):=2\gamma t-f(t)[1+(2\gamma)^{p-1}f(t)^{p(2\gamma-1)}]^{1/p}$, $t \geq 0$ and note that  $F_1(0)=0$ and
$$
\begin{array}{lll}
 F_1'(t)&=&\displaystyle(2\gamma-1)-(2\gamma-1)(2\gamma)^{p-1}f(t)^{p(2\gamma-1)}[1+(2\gamma)^{p-1}f(t)^{p(2\gamma-1)}]^{\frac{1-p}{p}}f'(t)\\
 \\
&=&\displaystyle\frac{2\gamma-1}{1+(2\gamma)^{p-1}f(t)^{p(2\gamma-1)}}=(2\gamma-1)f'(t)^p,~t>0.
\end{array}
$$
So, it follows from $(f)_2$ that $F_1'(t)>0,~t>0$, that is, the first inequality follows from the non-negativeness of $F_1$.

In a similar way, defining $F_2(t):=t-f(t)[1+(2\gamma)^{p-1}f(t)^{p(2\gamma-1)}]^{1/p}$, $t \geq0$ 
and noting that $F_2(0)=0$ and 
$$
\begin{array}{ccc}
 F_2'(t)
&=&-(2\gamma-1)(2\gamma)^{p-1}f(t)^{p(2\gamma-1)}[1+(2\gamma)^{p-1}f(t)^{p(2\gamma-1)}]^{\frac{1-p}{p}}f'(t)<0,~t>0,
\end{array}
$$
it follows our second inequality. 

\noindent $Proof~of~(f)_7:$ Since, 
$$
\begin{array}{ccc}
 \displaystyle\Big( \displaystyle\frac{f(t)}{t^{1/2\gamma}}\Big)'&=& \displaystyle\frac{f'(t)t^{1/2\gamma}-\frac{1}{2\gamma}t^{\frac{1}{2\gamma}-1}f(t)}{t^{1/\gamma}}
= \displaystyle\frac{2\gamma tf'(t)-f(t)}{2\gamma t t^{1/2\gamma}},
\end{array}
$$
it follows from $(f)_6$, that 
 ${f(t)}/{t^{1/2\gamma}},~t>0$ is increasing. So, by using  $(f)_5$, we obtain our claim.

\noindent $Proof~of~(f)_8:$ It follows from $(f)_4$,  $(f)_7$, and the fact of $f$ being odd, that 
$$
|f(t)|\geq \left\{
\begin{array}{ll}
 C|t|,& |t|\leq1,\\
C|t|^{1/2\gamma},& |t|\geq1,
\end{array}
\right.
$$
for some real positive constant  $C$. That is,
$|t|\leq C[|f(t)|+|f(t)|^{2\gamma}]$ for all $t\in\mathbb{R}.$

\noindent $Proof~of~(f)_9:$ This is an immediate   consequence of the definition of $f$.

\noindent $Proof~of~(f)_{10}:$ It follows from definition, that
$$f''(t)=-(2\gamma)^{p-1}(2\gamma-1)f(t)^{p(2\gamma-1)-1}[1+(2\gamma)|f(t)|^{p(2\gamma-1)}]^{-{2}/{p}-1}~t\geq 0,$$
that is,
$$(f'(t)f(t)^\delta)'=[1+(2\gamma)|f(t)|^{p(2\gamma-1)}]^{-{2}/{p}}f(t)^{\delta-1}
\frac{\delta+(2\gamma)^{p-1}(\delta+1-2\gamma)f(t)^{p(2\gamma-1)}}{1+(2\gamma)^{p-1}|f(t)|^{p(2\gamma-1)}}>0,~t>0,$$
because $\delta \geq 2 \gamma -1$, by hypothesis. These end our proof.
 \end{proof}
\fim

Below, we are going to apply the the function $f$ determined by (\ref{mv}) to reduce (\ref{P}) to another one. So, we have.

\begin{lemma}\label{lemma2} 
Assume $u=f(w)$ $($or $w=f^{-1}(u)$$)$. Then  $u$ is a solution of (\ref{P}) if, and only if $w$ is a solution of the problem 
 \begin{equation}\label{PP}
 \left\{
\begin{array}{c}
 \Delta_pw=a(x)g(f(w))f'(w)\ \ \ in \ \mathbb{R}^N,\\
w\geq0\  in\ \mathbb{R}^N,\ \  w(x)\stackrel{\left|x\right|\rightarrow \infty}{\longrightarrow} \infty.
\end{array}
\right.
\end{equation}
\end{lemma}
\begin{proof} First, note that it follows from Lemma \ref{lemma1}-$(f)_1$ that $u\in C^1(\mathbb{R}^N)$ if, and only if, $w\in C^1(\mathbb{R}^N)$ and $u \geq 0$ if, and only if, $w\geq 0$. Besides this, it follows from Lemma \ref{lemma1}-$(f)_5$ and $(f)_8$ that
$w(x)\to\infty$ as $|x|\to\infty$ if, and only if, $u(x)\to\infty$ as $|x|\to\infty$. To complete the proof,  since
 $$\nabla u=f'(w)\nabla w=\frac{1}{[1+(2\gamma)^{p-1}|f(w)|^{p(2\gamma-1)}]^{\frac{1}{p}}}\nabla w=
\frac{1}{[1+(2\gamma)^{p-1}|u|^{p(2\gamma-1)}]^{\frac{1}{p}}}\nabla w,$$
we get to
$$(1+(2\gamma)^{p-1}|u|^{p(2\gamma-1)})|\nabla u|^{p-2}\nabla u=[1+(2\gamma)^{p-1}|u|^{p(2\gamma-1)}]^{\frac{1}{p}}
|\nabla w|^{p-2}\nabla w,$$
that is,
$$
\int_{\mathbb{R}^N}(1+(2\gamma)^{p-1}|u|^{p(2\gamma-1)})|\nabla u|^{p-2}\nabla u\nabla\varphi dx=
\int_{\mathbb{R}^N}[1+(2\gamma)^{p-1}|u|^{p(2\gamma-1)}]^{\frac{1}{p}}|\nabla w|^{p-2}\nabla w\nabla\varphi dx,
$$
for all $\varphi\in C_0^1(\mathbb{R}^N)$.

On the other side, since 
 $$[1+(2\gamma)^{p-1}|u|^{p(2\gamma-1)}]^{1/p}\nabla\varphi=\nabla\{[1+(2\gamma)^{p-1}|u|^{p(2\gamma-1)}]^{1/p}\varphi\}-
\nabla[1+(2\gamma)^{p-1}|u|^{p(2\gamma-1)}]^{1/p}\varphi,$$
holds, we obtain that
$$
\begin{array}{lll}
\displaystyle \int_{\mathbb{R}^N}(1+(2\gamma)^{p-1}|u|^{p(2\gamma-1)})|\nabla u|^{p-2}\nabla u\nabla\varphi dx&=&\displaystyle\int_{\mathbb{R}^N}
|\nabla w|^{p-2}\nabla w\nabla\{[1+(2\gamma)^{p-1}|u|^{p(2\gamma-1)}]^{1/p}\varphi\}dx\\
\\
&-&\displaystyle\int_{\mathbb{R}^N}|\nabla w|^{p-2}\nabla w\nabla[1+(2\gamma)^{p-1}|u|^{p(2\gamma-1)}]^{1/p}\varphi dx\\
\\
&=&-\displaystyle\int_{\mathbb{R}^N}[a(x)g(u)+(2\gamma)^{p-1}(2\gamma-1)|\nabla u|^p|u|^{p(2\gamma-1)-1}]\varphi dx.
\end{array}
$$
for all $\varphi\in C_0^1(\mathbb{R}^N)$. So, by density, it follows our proof.
\end{proof}

\section{Proof of Theorems}

In the sequel, we are going to apply the last two Lemmas, together with some ideas found in \cite{mahammed}, to complete our proof.

\noindent\begin{proof}$of~ Theorem~ \ref{THRN}$. Since $a(x)=a(\vert x \vert)$, $x \in \mathbb{R}^N$, we have that (\ref{PP}) is equivalent to the problem
\begin{equation}\label{w}
 \left\{ 
\begin{array}{l}
 (r^{N-1}|w'|^{p-2}w')'=r^{N-1}a(r)g(f(w))f'(w)\ \ \ in \ (0,\infty),\\
w'(0)=0,\ \ w(0)=\alpha\geq 0,
\end{array}
\right.
\end{equation}
where $r=\vert x \vert \geq 0$ and $\alpha\geq 0$ is a real number.  Now, since $a$, $g$ and $f'$ are continuous functions, it follows from an approach in \cite{haitao} that there exists a  
$\Gamma(\alpha)>0$ (maximal extreme to the right for the existence interval of solutions for (\ref{w}), and a 
$w_\alpha\in C^2(0,\Gamma(\alpha))\cap C^1([0,\Gamma(\alpha)))$ solution of (\ref{w}) on $(0,\Gamma(\alpha))$, for each $\alpha> 0$ given.

If we assumed that $\Gamma(\alpha)<\infty$ for some $\alpha>0$, we obtain by standard arguments on ordinary differential equations that $w_\alpha(r)\to\infty$ as $r\to\Gamma(\alpha)^-$, that is, $w_\alpha(|x|)$ would satisfies to the problem
$$
 \left\{ 
\begin{array}{l}
 (r^{N-1}|w'|^{p-2}w')'=r^{N-1}a(r)g(f(w))f'(w)\ \ \ in \ (0,\infty),\\
w(0)=\alpha>0,~w'(0)=0,\  w_\alpha(x)\stackrel{r\rightarrow \Gamma(\alpha)^-}{\longrightarrow} \infty.
\end{array}
\right.
$$

Moreover, it follows from Lemma \ref{lemma1}-$(f)_2$,  that $w$ satisfies
\begin{equation}\label{ww}
 \left\{ 
\begin{array}{l}
 (r^{N-1}|w'|^{p-2}w')'\leq a_\infty r^{N-1}g(f(w))\ \ \ in \ (0,\infty),\\
w(0)=\alpha>0,~w'(0)=0,\  w_\alpha(x)\stackrel{r\rightarrow \Gamma(\alpha)^-}{\longrightarrow} \infty,
\end{array}
\right.
\end{equation}
where  $a_\infty=\max_{\bar{B}_{\Gamma(\alpha)}}a(x)$.

Since $u'\geq 0$, we can rewrite the inequality in (\ref{ww}) as 
$$\Big( (w')^{p-1}\Big)' \leq a_{\infty} (g\circ f)(w),~\mbox{for all}~0<r<\Gamma(\alpha) $$
and multiplying by ${w'}$ and integrating on $(0,r)$, we obtain 
$$\displaystyle\frac{p-1}{p}\Big( w'(r)\Big)^p \leq \int_0^r  a_{\infty} (g\circ f)(w(s))w'(s) ds = a_{\infty}\int_0^{w(r)}(g\circ f)(s) ds - a_{\infty}\int_0^{w(0)}(g\circ f)(s) ds,
$$
that is,
 $${\Big(\int_0^{w(r)}(g\circ f)(s) ds \Big)^{-1/p}{w'(r)}}\leq \Big(a_{\infty}\frac{p}{p-1}\Big)^{1/p},~\mbox{for all}~0<r<\Gamma(\alpha).$$
 
Now, by integrating in the last inequality over $(0,\Gamma(\alpha))$ and reminding that $w_\alpha(x){\to} \infty$ as ${r\rightarrow \Gamma(\alpha)^-}$, we obtain
 
\begin{eqnarray}
\label{comp1}
\int_{w(0)}^\infty\Big(\int_0^t(g\circ f)(s)ds\Big)^{-{1}/{p}}dt \leq \Big(a_{\infty}\frac{p}{p-1}\Big)^{1/p} \Gamma(\alpha) < \infty.
\end{eqnarray}

On the other side, it follows from Lemma \ref{lemma1}-$(f)_3$, and monotonicity of $g$, that
$(g\circ f)(t))\leq g(t)$ for all $t\geq0$. That is,  
$$\int_0^t(g\circ f)(s)ds\leq\int_0^tg(s)ds=G(t),~t \geq 0.$$

As a consequence of this, we have
$$ \int_{w(0)}^\infty G(t)^{-{1}/{p}}dt\leq \int_{w(0)}^\infty\Big(\int_0^t(g\circ f)(s)ds\Big)^{-{1}/{p}}dt.$$
So, it follows from (\ref{comp1}), that
$$ \int_{1}^\infty G(t)^{-{1}/{p}}dt< \infty,$$
but this is impossible, because we are assuming that $g$ satisfies the hypothesis {\bf(G)}.

To complete the proof,  it follows from 
$(f)_3$, $(f)_7$ and of  definition of $f$, that there exist real constants $A_1, A_2>0$ such that
$$f(t)\geq A_1t^{1/2\gamma}\ \ \mbox{and}\ \ f'(t)\geq A_2t^{1-2\gamma}\ \mbox{for all}\ \ t>\mathcal{A},$$
for some $\mathcal{A}>0$. So, as a consequence of these, we have
$$
\begin{array}{lll}
 \displaystyle g(f(w_\alpha(r)))f'(w_\alpha(r))&\geq&A_2g(f(w_\alpha))w_\alpha^{1-2\gamma}=A_2\frac{g(f(w_\alpha))}{f(w\alpha)^{2\gamma(2\gamma-1)}}[\frac{f(w_\alpha)^{2\gamma}}{w_\alpha}]^{(2\gamma-1)}\\
\\
&\geq&\displaystyle A_2A_1^{(2\gamma-1)}\frac{g(f(w_\alpha))}{f(w_\alpha)^{2\gamma(2\gamma-1)}}:=M>0,~\mbox{for all}~r>0,
\end{array}
$$
and for each $\alpha>\mathcal{A}$ given, because $w_\alpha(r) \geq \alpha$, for all $r\geq 0$.

Since, $w_\alpha$ satisfies
$$w_\alpha(r)=\alpha+\int_0^r\Big(t^{1-N}\int_0^ts^{N-1}a(s)g(f(w_\alpha))f'(w_\alpha)ds\Big)^{1/p-1}dt,~r\geq 0,$$
it follows from above informations, that
$$w_\alpha(t)\geq\alpha+M^{1/p-1}\int_0^r\Big(t^{1-N}\int_0^rs^{N-1}a(s)ds\Big)^{1/p-1}dt\to\infty,\ \mbox{when}\ r\to\infty.$$
This end the proof.
\end{proof}
\fim

\noindent\begin{proof}$of~ Theorem~ \ref{T2}$.
Given $\beta>\alpha>\mathcal{A}$, where $\mathcal{A}>0$ was given above,  it follows from  Theorem \ref{THRN} that there exist positive and radial solutions $w_\alpha$ and $w_\beta$  to 
the problems
$$
\left\{
\begin{array}{l}
 \Delta_pw_\alpha=\overline{a}(|x|)g(f(w_\alpha))f'(w_\alpha)\ in~\mathbb{R}^N,\\
w_\alpha(0)=\alpha,~ w_\alpha(x)\stackrel{\left|x\right|\rightarrow \infty}{\longrightarrow} \infty,
\end{array}
\right.~~~~\mbox{and}~~~~\left\{
\begin{array}{l}
 \Delta_pw_\beta=\underline{a}(|x|)g(f(w_\beta))f'(w_\beta)\ in~\mathbb{R}^N,\\
w_\beta(0)=\beta,~ w_\beta(x)\stackrel{\left|x\right|\rightarrow \infty}{\longrightarrow} \infty,
\end{array}
\right.
$$
respectively, where $\underline{a}$ and $\overline{a}$ were defined in (\ref{defa}).

Besides this, it follows from $w_\alpha$ and $g$ be nondecreasing,  (\ref{a}) and Lemma \ref{lemma1}-$(f)_3$, that
$$
\begin{array}{lll}
 w_\alpha(r)&\leq&2g(w_\alpha(r))^{1/(p-1)}\displaystyle\int_0^r \Big(\int_0^t\overline{a}(s)ds\Big)^{1/p-1}dt\\
&\leq&\displaystyle2g(w_\alpha(r))^{1/(p-1)}\Big[r\Big(\int_0^r\overline{a}(t)dt\Big)^{1/p-1}-\frac{1}{p-1}\int_0^rt\overline{a}(t)
\Big(\int_0^t\overline{a}(s)ds\Big)^{2-p/p-1}dt\Big] \\
&\leq&\displaystyle2rg(w_\alpha(r))^{1/(p-1)}\Big(\int_0^r\overline{a}(t)dt\Big)^{1/p-1},
\end{array}
$$
for all $r>0$ sufficiently large. That is, 
$$
w_\alpha(r)\leq \mathcal{G}^{-1}\Big(r\Big(\int_0^r \overline{a}(t)dt\Big)^{1/p-1}\Big),~\mbox{for all}~r>>0.
$$

Now, setting 
$$0<S(\beta)=sup\{r>0 ~/~w_\alpha(r)<w_\beta(r)\} \leq \infty,$$
 we claim that $S(\beta)=\infty$ for all $\beta>\alpha + \overline{H}$, for each $\alpha>\mathcal{A}$ given. In fact, by assuming this is not 
true, then there exists a $\beta_0>\alpha + \overline{H}$ such that $w_\alpha(S(\beta_0))=w_\beta(S(\beta_0))$. So, by using that $g$ is 
non-decreasing, Lemma \ref{lemma1}-$(f)_{10}$ and $w_{\alpha} \leq w_{\beta}$ on $[0,S(\beta_0)]$, we obtain that  
\begin{equation}\label{wh}
\begin{array}{ll}
\beta_0 &= \displaystyle\alpha +\int_0^{S(\beta_0)}\Big[\Big(t^{1-N}\int_0^ts^{N-1}\overline{a}(s)g(f(w_\alpha(s)))f'(w_\alpha)ds\Big)^{1/p-1}\\
&-\displaystyle\Big(t^{1-N}\int_0^ts^{N-1}\underline{a}(s)g(f(w_\beta(s)))f'(w_\beta)ds\Big)^{1/p-1}\Big]dt\\
&= \displaystyle\alpha +\int_0^{S(\beta_0)}\Big[\Big(t^{1-N}\int_0^ts^{N-1}\overline{a}(s)g(w_\alpha(s))f'(w_\alpha) ds\Big)^{1/p-1}\\
&-\displaystyle\Big(t^{1-N}\int_0^ts^{N-1}\underline{a}(s)\frac{g(f(w_\beta(s)))}{f(w_\beta)^\delta}f'(w_\beta)f(w_\beta)^\delta ds\Big)^{1/p-1}\Big]dt\\
&= \displaystyle\alpha +\int_0^{S(\beta_0)}\Big[\Big(t^{1-N}\int_0^ts^{N-1}\overline{a}(s)g(w_\alpha(s))f'(w_\alpha) ds\Big)^{1/p-1}\\
&-\displaystyle\Big(t^{1-N}\int_0^ts^{N-1}\underline{a}(s)g(f(w_\alpha(s)))f'(w_\alpha) ds\Big)^{1/p-1}\Big]dt
\end{array}
\end{equation}
holds.

On the other hands, it follows from $g,f$, and $w_\alpha$ being nondecreasing, that
$$
\begin{array}{lll}
0 &\leq& \Big[\Big(t^{1-N}\int_0^ts^{N-1}\overline{a}(s)g(f(w_\alpha))f'(w_\alpha)ds\Big)^{1/p-1}-
\Big(t^{1-N}\int_0^ts^{N-1}\underline{a}(s)g(f(w_\alpha))f'(w_\alpha)ds\Big)^{1/p-1}\Big] \chi_{[0,S(\beta)]}(t)\\
&=&\Big[\Big(t^{1-N}\int_0^ts^{N-1}[\overline{a}(s)-\underline{a}(s)]g(f(w_\alpha))f'(w_\alpha)ds+\Big(t^{1-N}\int_0^ts^{N-1}
\underline{a}(s)g(f(w_\alpha))f'(w_\alpha)ds\Big)\Big)^{1/p-1}\\
&-&\Big(t^{1-N}\int_0^ts^{N-1}\underline{a}(s)g(f(w_\alpha))f'(w_\alpha)ds\Big)^{1/p-1}\Big] \chi_{[0,S(\beta)]}(s)
\leq\Big(t^{1-N}\int_0^ts^{N-1}a_{osc}(s)g(f(w_\alpha))f'(w_\alpha)ds\Big)^{1/p-1}\\
&\leq&\Big(t^{1-N}\int_0^ts^{N-1}a_{osc}(s)ds\Big)^{1/p-1}\Big[ g\Big(\mathcal{G}^{-1}\Big(t\Big(\int_0^t\overline{a}(s)ds\Big)^{1/p-1}\Big)\Big)\Big]^{1/p-1}:
=\mathcal{H}(t) ,~t>>0,
\end{array}
$$
where $\chi_{[0,S(\beta)]}$ stands for the characteristic function of $[0,S(\beta)]$.

So, it follows from the hypothesis  $\bf{(\mathcal{G})}$ and  (\ref{wh}), that 
$$\beta_0 \leq \alpha +\int_0^\infty \mathcal{H}(s)ds\leq \alpha+ \overline{H},$$
but this is impossible. 

Now, by setting $\beta=(\alpha + \epsilon) + \overline{H}$, for each $\alpha>\mathcal{A},\epsilon>0$ given, and considering the problem
\begin{equation}\label{Wn}
 \left\{
\begin{array}{l}
 \Delta_p w={a}(x)g(f(w))f'(w)\ \mbox{in}\ B_n(0),\\
w\geq 0~\mbox{in}~B_n(0),~~w=w_\alpha\ \mbox{on}~ \partial B_n(0),
\end{array}
\right.
\end{equation}
we can infer by standard methods of sub and super solutions that there exists a $w_n = w_{n,\alpha} \in C^1(\overline{B}_n)$ solution of  
(\ref{Wn}) satisfying 
$A<\alpha\leq w_\alpha\leq w_n\leq w_\beta$ in $B_n$ for all $n \in \mathbb{N}$. So, by compactness, there exists a  
$w\in C^1(\mathbb{R}^N)$ such that $w(x)=\lim_{n\to\infty}w_n(x)$ is a  solution of (\ref{P}).\fim
\end{proof}
\begin{center}
\LARGE{Acknowledgement}
\end{center}

This paper was completed while the first author was visiting the Professor Haim Brezis at Rutgers University. He thanks to 
Professor Brezis by his incentive and hospitality. In this time, the second author was visiting the Professor Zhitao Zhang at 
Chinese Academy of Sciences. He is grateful by his invitation and 
hospitality as well.

\end{document}